\documentclass[preprint]{imsart}

\usepackage{amssymb}
\usepackage{amsmath}
\usepackage{amsthm}

\startlocaldefs
\numberwithin{equation}{section}
\theoremstyle{plain}
\newtheorem{lem}{Lemma}[section]
\newtheorem{df}{Definition}[section]
\newtheorem{tw}{Theorem}[section]

\newtheorem{rem}{Remark}[section]

\newtheorem{co}{Corollary}[section]
\newtheorem{cond}{Condition}[section]

\newcommand{\R}{\mathbb{R}}
\newcommand{\N}{\mathbb{N}}
\endlocaldefs

\begin{document}

\begin{frontmatter}

\title{The long-term behavior of number \\ of near-maximum insurance claims}
\runtitle{Asymptotics of number of near-maximum claims}

\begin{aug}
\author{\fnms{Anna} \snm{Dembi\'{n}ska\thanks{Corresponding author}}\corref{}\ead[label=e1]{dembinsk@mini.pw.edu.pl}}
\and
\author{\fnms{Aneta} \snm{Buraczy\'{n}ska}\ead[label=e2]{a.buraczynska@mini.pw.edu.pl}}

\runauthor{A. Dembi\'{n}ska and A. Buraczy\'{n}ska}

\affiliation{Warsaw University of Technology}

\address{
Faculty of Mathematics and Information Science\\
Warsaw University of Technology\\
ul. Koszykowa 75, 00-662 Warsaw, Poland\\
\printead{e1}
\phantom{E-mail:\ }\printead*{e2}}

\end{aug}

\begin{abstract}
A near-maximum insurance claim is one falling within a distance $a$ of the current maximal claim. 
In this paper, we investigate asymptotic behavior of normalized numbers of near-maximum insurance claims under the assumption  that the sequence of successive claim sizes forms a~strictly stationary process. We present the results in a general form expressing limiting properties of normalized numbers of insurance claims that are in a left neighborhood of the $m_n$th largest claim, where $m_n/n$ tends to zero and $n$ is the number of registered claims. We also give corollaries for sums of 
near-maximum insurance claims.

\end{abstract}



\end{frontmatter}

\section{Introduction}
\label{sec1}

Various models of claim exceedances over a fixed or random threshold have been studied in the literature. Special attention has been paid to models of  exceedances of high thresholds, in particular of thresholds determined by the largest or the $k$th largest claim in the portfolio; see, for example, Embrechts et al. (1997),  Li and Pakes (2001),  Hashorva (2003), Chavez-Demoulin and  Embrechts (2004), Balakrishnan et al. (2005), Eryilmaz et al. (2011), Bose and Gangopadhyay (2011), and the references therein. Such models are of special interest in actuarial sciences - they serve to measure, predict and price risk. In different contexts, they also proved to be useful in a variety of other fields like hydrology, environmental research or electrical engineering (Embrechts et al. 1997, Chapter 6).

In this paper, we consider a model of  exceedances of  insurance claims introduced by  Li and Pakes (2001) and studied also  by Hashorva (2003). In this model the claims occur at the random instants of time and, for $t\geq 0$, the number of claims registered during the time interval $[0,t]$ is denoted by $N(t)$. Clearly it is required that $N(t), t\geq 0$, are non-negative and integer-valued random variables (rv's). Sizes of successive claims are represented by rv's $X_1,X_2,\ldots$ and $X_{1:n}\leq X_{2:n}\leq\ldots\leq X_{n:n}$ denote the order statistics corresponding to $X_1,X_2,\ldots,X_n$. Then $X_i$ is called a near maximum claim at time $t$ if its value falls within a distance $a>0$ of the current maximal claim size, that is if
$$X_i\in(X_{N(t):N(t)}-a,X_{N(t):N(t)}).$$
A quantity of interest is the number of near-maximum claims up to time $t$:
$$
{\cal K}_t(a)=\sum_{i=1}^{N(t)} I(X_{N(t):N(t)}-a<X_i<X_{N(t):N(t)}).
$$
 Li and Pakes (2001) derived some distributional and asymptotic results for ${\cal K}_t(a)$ under the assumption that $X_1,X_2,\ldots$ are independent and identically distributed continuous rv's and the random processes $(X_n, n\geq 1)$ and $(N(t), t\geq 0)$ are independent. In particular, they pointed out that the limiting behavior of  ${\cal K}_t(a)$ as $t\to\infty$ and the following rv  counting near-maxima
$$
K_n(a)=\sum_{i=1}^nI(X_{n:n}-a<X_i<X_{n:n})
$$
as $n\to\infty$ are closely related under a weak condition on $N(t)$. They also showed how asymptotic properties of ${\cal K}_t(a)$  as $t\to\infty$ can be deduced from these for $K_n(a)$ as $n\to\infty$. Their method was further developed and exploited by  Hashorva (2003) in the case when some kind of dependence between  claim sizes is allowed.

The aim of this paper is to describe asymptotic behavior of numbers of near-maximum claims under a general assumption that claim sizes $X_1,X_2,\ldots$  form a~strictly stationary process. We will show that this behavior remains unchanged if instead of numbers of near-maximum claims we will consider numbers of claims  in a left neighborhood of  extreme or intermediate order statistic:
\begin{equation}
\label{defcalK}
{\cal K}_t(k_{N(t)},a)=\sum_{i=1}^{N(t)} I(X_{k_{N(t)}:N(t)}-a<X_i<X_{k_{N(t)}:N(t)}),
\end{equation}
where $(k_n,n\ge 1)$ is a~sequence of  integers satisfying the following condition
\begin{equation}
\label{warK}
\textrm{for all }n\ 1\le k_n\le n\textrm{ and }\lim_{n\to\infty} k_n/n=1.
\end{equation}
To describe this behavior we will first focus on the deterministic-sample-size analogue of ${\cal K}_t(k_{N(t)},a)$, that is on the following counting rv
$$
K(k_n,n,a)=\sum_{i=1}^n I( X_{k_n:n}-a<X_i<X_{k_n:n}),
$$
where again $(k_n,n\ge 1)$  satisfies \eqref{warK}. Our crucial result will assert that under mild conditions proportions $K(k_n,n,a)/n$  are almost surely convergent as $n\to\infty$. The most important novelty of this result is that in the strictly-stationary case the limit of $K(k_n,n,a)/n$  can be a non-degenerate rv in contrast to   cases known in the literature  in which the limit is a deterministic  constant.

We will restrict our attention to asymptotic behavior of  numbers of claims near extreme or intermediate order statistic. That of numbers of claims near central order statistic, when in \eqref{defcalK} we have a sequence  $(k_n,n\ge 1)$ satisfying 
$$
\textrm{for all }n\ 1\le k_n\le n\textrm{ and }\lim_{n\to\infty} k_n/n=\lambda\in(0,1),
$$
is quite well known. For details we refer to the works of  Dembi\'nska et al. (2007),  Pakes (2009), Dembi\'nska and Jasi\'nski (2017) and  Dembi\'nska (2012a, 2017).

The paper is organized as follows. Section~\ref{Sec2} presents some preliminaries.
In Section~\ref{sec3}, we describe almost sure limiting behavior of $K(k_n,n,a)/n$ as $n\to\infty$, where $(k_n, n\geq 1)$ is a sequence satisfying  \eqref{warK}. We present this description under a quite general assumption that the underlying sequence of rv's $(X_n,n\ge~1)$ forms a strictly stationary process.
Next, in Section~\ref{Sec4}, we generalize results of the previous section to the case of randomly indexed samples. In particular, we provide description of limiting behavior of appropriately normalized numbers of near-maximum insurance claims. We also give comments concerning the total value of such claims. Finally, in Section~\ref{sec:Examples}, we apply our  results to some special families of strictly stationary processes.

Throughout the paper  we make use of the following notation.  
By $\mathbb{R}$ we denote the set of real numbers. 
For an rv $X$ with  cumulative distribution function~$F$, we set 
$$ \gamma_1^X:=\sup\{x\in\mathbb{R}:\ F(x)<1\}$$
and call $ \gamma_1^X$ the right endpoint of the support of $X$. 
We write  $I(A)$ for the indicator function of a set $A\subset\Omega$, i.e., for $\omega\in\Omega$, $I(A)(\omega)=1$ if $\omega\in A$ and $I(A)(\omega)=0$ otherwise. 
The notation $\xrightarrow{a.s}$, $\xrightarrow{p}$, $\xrightarrow{d}$ stand for almost sure convergence, convergence in probability and  convergence in distribution, respectively. The symbol $a.s$  is an~abbreviation of almost surely. Moreover, in the case when different probability measures appear and confusion may arise, we write $\xrightarrow{\mathbb{P}-a.s}$ and $\mathbb{E}_{\mathbb{P}}$ for almost sure convergence and expectation with respect to the measure $\mathbb{P}$, respectively, and  say that an~event A holds $\mathbb{P}$-a.s. if $\mathbb{P}(A)=1$.
Next,  $a$~is always a fixed positive real number. Any empty sum $\sum_{i=1}^0 a_i$ is by assumption equal to $0$. 
Finally, an extended rv in $(\Omega, \mathcal F, \mathbb{P})$ is a $\mathcal F$-measurable function $X: \Omega\to[-\infty,\infty]$. We adopt standard conventions about arithmetic operations and order relation in $[-\infty,\infty]$. In particular,  if $x\in\R$ then $x\pm\infty=\pm\infty$ and $x<\infty$, and  if $x>0$, then $x\cdot\infty=\infty$.


\section{Preliminaries}
\label{Sec2}

For our developments we will need the concept of conditional right endpoint of the support of an rv given a~sigma-field.
Here, we recall its definition and some of its properties that will be used later on. 

\begin{df}\label{kwantyl}
Suppose $X$ is an~rv on a~probability space $(\Omega,\mathcal{F},\mathbb{P})$,
     and $\mathcal{G}$ is a~sigma-field with $\mathcal{G}\subseteq\mathcal{F}$. The conditional right endpoint of~the support of~$X$ given $\mathcal{G}$, denoted by $\gamma_1(X|\mathcal{G})$, is defined as an~extended rv $\underline{Q}_1$ with the following properties 
    \begin{description}
    \item[(i)]
    $\underline{Q}_1$ is $\mathcal{G}$-measurable,
    \item[(ii)]
    $\mathbb{P}(X\le \underline{Q}_1|\mathcal{G})=1$ a.s.,
    \item[(iii)]
    for any $\mathcal{G}$-measurable extended rv~$Q_1$ such that $\mathbb{P}(X\le Q_1|\mathcal{G})=1$ a.s., we have $Q_1\ge \underline{Q}_1$ a.s.
    \end{description}
\end{df}

The theorem given below ensures the existence of conditional right endpoint of the support of an rv.
\begin{tw}
    For any rv $X$ on $(\Omega,\mathcal{F},\mathbb{P})$ and any sigma-field $\mathcal{G}\subseteq\mathcal{F}$ there exists a~conditional right endpoint of~the support of~rv~$X$ given $\mathcal{G}$.
\end{tw}
Moreover, for  given rv $X$ and  sigma-field~$\mathcal{G}$, $\gamma_1(X|\mathcal{G})$  is not necessarily uniquely determined,  but any two versions of $\gamma_1(X|\mathcal{G})$  are equal almost surely.

The next theorem gives another worthwhile property of conditional right endpoint of~the support.
\begin{tw}
\label{stalyKwantyl}
Let $X$ be an rv and $\mathcal{G}\subseteq\mathcal{F}$ be a~sigma-field. If $\gamma_1(X|\mathcal{G})$ is almost surely constant, then $\gamma_1(X|\mathcal{G})=\gamma_1^{X}$ a.s.
\end{tw}
For more properties of conditional right endpoint of~the support of an rv and proofs of the above-mentioned facts, we refer the reader to~Buraczy\'nska and Dembi\'nska (2018).

\medskip

To state and prove main results of this paper we will also need 
 some terminology and facts from the ergodic theory, which we now recall. 
By $(\mathbb{R}^{\mathbb{N}},\mathcal{B}(\mathbb{R}^{\mathbb{N}}),\mathbb{Q})$ we denote a~probability triple, where $\mathbb{R}^{\mathbb{N}}$ is the set of sequences of real numbers $(x_1,x_2,\ldots)$, $\mathcal{B}(\mathbb{R}^{\mathbb{N}})$ stands for the Borel sigma-field of subsets of $\mathbb{R}^{\mathbb{N}}$ and $\mathbb{Q}$ is some stationary probability measure on the pair $(\mathbb{R}^{\mathbb{N}},\mathcal{B}(\mathbb{R}^{\mathbb{N}}))$.
We will  use three types of invariant sets.

\begin{df} Let, for $A\subset \mathbb{R}^{\mathbb{N}}$, 
$$T^{-1}A=\{(x_1,x_2,\ldots)\in\mathbb{R}^{\mathbb{N}}:\ (x_2,x_3,\ldots)\in A \}.$$
A~set $A\in \mathcal{B}(\mathbb{R}^{\mathbb{N}})$ is called 
\begin{itemize}
\item  invariant if $A=T^{-1}A$;
\item almost invariant for $\mathbb{Q}$ if 
    \[
    \mathbb{Q}((A\setminus T^{-1}A)\cup(T^{-1}A\setminus A))=0;
    \]
\item  invariant with respect to the sequence $\mathbb{X}=(X_n,n\ge 1)$ defined on~the probability space $(\Omega,\mathcal{F},\mathbb{P})$ if there exists a~set $B\in\mathcal{B}(\mathbb{R}^{\mathbb{N}})$ such that
\begin{equation}\label{warAinvar}
A=\{\omega\in\Omega:\ (X_i(\omega),X_{i+1}(\omega),\ldots)\in B \}\textrm{ for any }i\ge 1.
\end{equation}
\end{itemize}
\end{df}
By $\tilde{\mathcal{I}}$ we denote the collection of all invariant events, while $\mathcal{I}_{\mathbb{Q}}$ and $\mathcal{I}^{\mathbb{X}}$ correspond to
the class of all almost invariant sets for~$\mathbb{Q}$ and 
the class of all invariant sets with respect to the sequence~$\mathbb{X}$, respectively.
The following lemma states  significant properties of these classes. 

\begin{lem}
\label{l2}
Let $\mathbb{X}=(X_n,n\ge 1)$ be a~strictly stationary sequence on $(\Omega,\mathcal{F},\mathbb{P})$.
\begin{description}
\item[(i)]
    $\tilde{\mathcal{I}}$, $\mathcal{I}_{\mathbb{Q}}$ and $\mathcal{I}^{\mathbb{X}}$ are sigma-fields.
\item[(ii)]
    An rv $X$ on $(\mathbb{R}^{\mathbb{N}},\mathcal{B}(\mathbb{R}^{\mathbb{N}}),\mathbb{Q})$ is
\begin{itemize}
\item  $\tilde{\mathcal{I}}$-measurable  if and only if  (iff)   
    \[
    X((x_1,x_2,\ldots))=X((x_2,x_3,\ldots))\textrm{ for  every }(x_1,x_2,\ldots)\in\mathbb{R}^{\mathbb{N}}.
    \]
\item  $\mathcal{I}_{\mathbb{Q}}$-measurable  iff 
    \[
    X((x_1,x_2,\ldots))=X((x_2,x_3,\ldots))\textrm{ for }\mathbb{Q}\textrm{-almost every }(x_1,x_2,\ldots)\in\mathbb{R}^{\mathbb{N}}.
    \]
\end{itemize}
\item[(iii)]
    A~set $A\in\mathcal{F}$ is invariant with respect to~$\mathbb{X}$ iff there exists a~set $B\in\mathcal{I}_{\mathbb{Q}}$ satisfying~\eqref{warAinvar}.
\end{description}
\end{lem}
Parts (i) and (ii) of Lemma \ref{l2} can be found, for example, in Shiryaev (1996, Chapter V) and Durrett (2010, Section 6). For part (iii) we refer the reader to Buraczy\'nska and Dembi\'nska (2018, Lemma 4.2).

We will also need the following lemma.
\begin{lem}
\label{bez.war}
Let  $(\Omega,\mathcal{F},\mathbb{P})$ be any probability space and  $\mathcal{G}\subseteq\mathcal{F}$ be a sigma-field. Then, for any $A\in\mathcal{F}$,
$$ 
\mathbb{P}(A|\mathcal{G})=0 \quad  \mathbb{P}-a.s.\; \Leftrightarrow \; \mathbb{P}(A)=0.
$$
\end{lem}
\begin{proof}
By the definition of conditional expectation, $\mathbb{P}(A|\mathcal{G})=0 \;\; \mathbb{P}-a.s.$ iff
$$\mathbb{P}(A\cap G)=0 \hbox{ for any } G\in  \mathcal{G}.$$
Taking $G=\Omega$ we get $\mathbb{P}(A)=0$, which proves the ``if'' part of the lemma. To show the ``only if'' part, note that $\mathbb{P}(A)=0$ implies, for all $G\in  \mathcal{G}$,
$$0\leq \mathbb{P}(A\cap G)\leq \mathbb{P}(A)=0,$$
which clerly forces $\mathbb{P}(A\cap G)=0$.
\end{proof}

We conclude this section with three theorems describing the almost sure limiting behavior of extreme and intermediate order statistics arising from strictly stationary sequences of rv's. They are taken from Buraczy\'nska and Dembi\'nska (2018) and will be used in Section \ref{sec3}. The first one concerns the situation when the stationary sequence of rv's is carried by a  probability triple $(\mathbb{R}^{\mathbb{N}},\mathcal{B}(\mathbb{R}^{\mathbb{N}}),\mathbb{Q})$.

\begin{tw}\label{tw.1}
Let $Y$ be an rv on a probability space $(\mathbb{R}^{\mathbb{N}},\mathcal{B}(\mathbb{R}^{\mathbb{N}}),\mathbb{Q})$, where the probability
measure $\mathbb{Q}$ is stationary. Suppose that the sequence of rv's $(Y_n, n \ge 1)$ is defined by
\begin{equation}\label{def_Y}
Y_i ((x_1, x_2,\ldots)) = Y ((x_i , x_{i+1},\ldots))\textrm{ for }(x_1, x_2,\ldots) \in \mathbb{R}^{\mathbb{N}}\textrm{ and }i \ge 1.
\end{equation}
If $(k_n, n \ge 1)$ is a~sequence of  integers satisfying~\eqref{warK}
then
\begin{equation}\label{teza1}
Y_{k_n:n}\xrightarrow{\mathbb{Q}-a.s.}\gamma_1(Y|\mathcal{I}_{\mathbb{Q}}).
\end{equation}
\end{tw}

The second result shows that  the above theorem remains true if the sigma-field~$\mathcal{I}_{\mathbb{Q}}$ is replaced by $\tilde{\mathcal{I}}$.

\begin{tw}
\label{tw.statystyki2.4}
If assumptions of Theorem~\ref{tw.1} hold, 
then
\begin{equation}\label{teza3}
Y_{k_n:n}\xrightarrow{\mathbb{Q}-a.s.}\gamma_1(Y|\tilde{\mathcal{I}}).
\end{equation}
\end{tw}

Finally, the third result is a generalization of Theorems \ref{tw.1} and \ref{tw.statystyki2.4}  to the case of strictly stationary sequences defined on any probability space.
\begin{tw}\label{tw.2}
Let $\mathbb{X}=(X_n,n\ge 1)$ be a~strictly stationary sequence on $(\Omega,\mathcal{F},\mathbb{P})$ and $(k_n,n\ge 1)$ be a~sequence of  integers satisfying~\eqref{warK}. Then
\begin{equation}\label{teza2}
X_{k_n:n}\xrightarrow{\mathbb{P}-a.s.}\gamma_1(X_1|\mathcal{I}^{\mathbb{X}}).
\end{equation}
\end{tw}


\section{The case of deterministic sample size}
\label{sec3}

As already mentioned in the Introduction, the asymptotic behavior of $K_n(a)$ and ${\cal K}_t(a)$ is closely linked. More generally, the asymptotic behavior of $K(k_n,n,a)$ and ${\cal K}_t(k_{N(t)},a)$ is closely related. Therefore to describe the long-term behavior of appropriately normalized numbers of near-maximum insurance claims, we first establish limiting results for $K_n(a)/n$. We present extended versions of these results examining not only the proportions of near maxima, $K_n(a)/n$, but more generally considering the proportions of observations  in a left neighborhood of  extreme or intermediate order statistic, $K(k_n,n,a)/n$, where the sequence $(k_n, n \ge 1)$  satisfies \eqref{warK}. More precisely, we show that if 
 $(X_n,n\ge 1)$ is a~strictly stationary process, then under some mild conditions the proportions $K(k_n,n,a)/n$ converge almost surely, and we describe the distribution of the limiting rv. To do this, 
 we first derive results for a special case of a specific sequence $\mathbb{Y}=(Y_n,n\ge 1)$ from the probability space $(\mathbb{R}^{\mathbb{N}},\mathcal{B}(\mathbb{R}^{\mathbb{N}}),\mathbb{Q})$ and then we extend these results to the case of strictly stationary sequences of rv's from an arbitrary probability space. 

To avoid confusion, throughout this section, we add a superscript to  $K(k_n,n,a)$ so that $K^\mathbb{X}(k_n,n,a)$ indicates that we consider the number of observations near the $k_n$th order statistic corresponding to the sequence $\mathbb{X}=(X_n,n\ge 1)$.

\begin{tw}\label{tw.3}
Under the assumptions of~Theorem~\ref{tw.1} and provided that
\begin{align}
   \label{zalozenia3}
    \mathbb{Q}\big(Y=\gamma_1(Y|\mathcal{I}_{\mathbb{Q}})\big)=0,\\
    \label{zalozenia2}
    \mathbb{Q}\big(Y=\gamma_1(Y|\mathcal{I}_{\mathbb{Q}})-a\big)=0,
\end{align}
we have
\begin{equation}
K^\mathbb{Y}(k_n,n,a)/n\xrightarrow{\mathbb{Q}-a.s.}\mathbb{Q}(Y>\gamma_1(Y|\mathcal{I}_{\mathbb{Q}})-a|\mathcal{I}_{\mathbb{Q}})\textrm{ as }n\to\infty.
\label{teza.tw.3}
\end{equation}
\end{tw}

\begin{proof}
For $m\in\N$, define the following rv $L_m$ and  possibly extended rv $R_m$,
\begin{align*}
L_m=
\begin {cases}
\gamma_1(Y|\mathcal{I}_{\mathbb{Q}})-\frac{1}{m} &\textrm{ if }\gamma_1(Y|\mathcal{I}_{\mathbb{Q}})<\infty,\\
m &\textrm{ if }\gamma_1(Y|\mathcal{I}_{\mathbb{Q}})=\infty,\\
\end{cases}
\end{align*}
and
\begin{align*}
R_m=
\begin {cases}
\gamma_1(Y|\mathcal{I}_{\mathbb{Q}})+\frac{1}{m} &\textrm{ if }\gamma_1(Y|\mathcal{I}_{\mathbb{Q}})<\infty,\\
\infty &\textrm{ if }\gamma_1(Y|\mathcal{I}_{\mathbb{Q}})=\infty.\\
\end{cases}
\end{align*}
By the definition of $\gamma_1(Y|\mathcal{I}_{\mathbb{Q}})$, $L_m$ and $R_m$ are $\mathcal{I}_{\mathbb{Q}}$-measurable.

Fix  $m\in\N$. Since from Theorem~\ref{tw.1}
\begin{equation}
Y_{k_n:n}\xrightarrow{\mathbb{Q}-a.s.}\gamma_1(Y|\mathcal{I}_{\mathbb{Q}}),
\label{przed3.5}
\end{equation}
we have, for all sufficiently large~$n$, 
$$L_m<Y_{k_n:n}\leq R_m \quad \mathbb{Q}-\textrm{a.s.},$$
which gives
\begin{align}
    K^\mathbb{Y}(k_n,n,a)/n&=\frac{\sum_{i=1}^n I(Y_{k_n:n}-a<Y_i<Y_{k_n:n})}{n} \nonumber \\
                                            &\ge\frac{\sum_{i=1}^n I(R_m-a<Y_i<L_m)}{n} \quad \mathbb{Q}-\textrm{a.s.}, \label{r1}
\end{align}
where the last inequality holds for all sufficiently large~$n$.
Set
\[
Z=I(R_m-a<Y<L_m)
\]
and 
$$
Z_i((x_1,x_2,\ldots))= Z((x_i,x_{i+1},\ldots))\textrm{ for all }(x_1,x_2,\ldots)\in\mathbb{R}^{\mathbb{N}}\textrm{ and }i\ge 1.
$$
Then, for $\mathbb{Q}-\textrm{almost every } (x_1,x_2,\ldots)\in\mathbb{R}^{\mathbb{N}}$,
\begin{align}
    Z_i((x_1,x_2,\ldots))&=I(R_m-a<Y<L_m)((x_i,x_{i+1},\ldots)) \nonumber \\
    &=I(R_m((x_i,x_{i+1},\ldots))-a<Y((x_i,x_{i+1},\ldots))<L_m((x_i,x_{i+1},\ldots))) \nonumber \\
    &=I(R_m((x_1,x_{2},\ldots))-a<Y_i((x_1,x_{2},\ldots))<L_m((x_1,x_{2},\ldots))) \nonumber \\
    &=I(R_m-a<Y_i<L_m)((x_1,x_{2},\ldots)), \label{po3.6}
\end{align}
where the third equality is a~consequence of~\eqref{def_Y}, the  $\mathcal{I}_{\mathbb{Q}}$-measurability of $L_m$ and $R_m$,  and the second part of Lemma~\ref{l2} (ii).
Hence, \eqref{r1} can be reformulated as follows
\begin{equation}
K^\mathbb{Y}(k_n,n,a)/n \ge \frac{1}{n}\sum_{i=1}^n Z_i\quad \mathbb{Q}-\textrm{a.s.}
\label{przed3.7}
\end{equation}
Now the classical ergodic theorem (see, for example, Durrett 2010, p. 333)  implies 
\begin{equation}
\label{ergodic_th}
\frac{1}{n}\sum_{i=1}^n Z_i\xrightarrow{\mathbb{Q}-a.s.}\mathbb{E}_{\mathbb{Q}}(Z|\mathcal{I}_{\mathbb{Q}})
=\mathbb{E}_{\mathbb{Q}}(I(R_m-a<Y<L_m)|\mathcal{I}_{\mathbb{Q}}) \quad \textrm{as }n\to\infty.
\end{equation}
From \eqref{przed3.7} and \eqref{ergodic_th} it may be concluded that, for all sufficiently large~$n$,
\begin{align*}
K^\mathbb{Y}(k_n,n,a)/n &\ge \mathbb{E}_{\mathbb{Q}}(I(R_m-a<Y<L_m)|\mathcal{I}_{\mathbb{Q}})-\frac{1}{m}\quad\mathbb{Q}-\textrm{a.s.}
\end{align*}
and in consequence that
\begin{equation}\label{r3}
\liminf_{n\to\infty}K^\mathbb{Y}(k_n,n,a)/n \ge
\mathbb{E}_{\mathbb{Q}}(I(R_m-a<Y<L_m)|\mathcal{I}_{\mathbb{Q}})-\frac{1}{m}\quad\mathbb{Q}-\textrm{a.s.}
\end{equation}
In an analogous way it can be proved that
\begin{equation}\label{r2}
\limsup_{n\to\infty}K^\mathbb{Y}(k_n,n,a)/n \le
\mathbb{E}_{\mathbb{Q}}(I(L_m-a<Y<R_m)|\mathcal{I}_{\mathbb{Q}})+\frac{1}{m}\quad\mathbb{Q}-\textrm{a.s.}
\end{equation}

Since
$$L_m\uparrow\gamma_1(Y|\mathcal{I}_{\mathbb{Q}}),\ R_m\downarrow \gamma_1(Y|\mathcal{I}_{\mathbb{Q}}) \quad \textrm{as }\;  m\to\infty,$$ 
we have
\[
I(R_m-a<Y<L_m)\xrightarrow{\mathbb{Q}-a.s.}I(\gamma_1(Y|\mathcal{I}_{\mathbb{Q}})-a<Y< \gamma_1(Y|\mathcal{I}_{\mathbb{Q}}))\quad \textrm{as }m\to\infty.
\]
Letting $m\to\infty$ in \eqref{r3} and using Dominated Convergence Theorem for conditional expectations, we get 
\begin{align*}
\liminf_{n\to\infty}
\ K^\mathbb{Y}(k_n,n,a)/n & \ge
\mathbb{E}_{\mathbb{Q}}(I(\gamma_1(Y|\mathcal{I}_{\mathbb{Q}})-a<Y<\gamma_1(Y|\mathcal{I}_{\mathbb{Q}}))|\mathcal{I}_{\mathbb{Q}})\\
&=\mathbb{Q}(\gamma_1(Y|\mathcal{I}_{\mathbb{Q}})-a<Y<\gamma_1(Y|\mathcal{I}_{\mathbb{Q}})|\mathcal{I}_{\mathbb{Q}} ) \\
&=\mathbb{Q}(Y> \gamma_1(Y|\mathcal{I}_{\mathbb{Q}}) -a|\mathcal{I}_{\mathbb{Q}})- \mathbb{Q}(Y= \gamma_1(Y|\mathcal{I}_{\mathbb{Q}})|\mathcal{I}_{\mathbb{Q}})\\
&=\mathbb{Q}(Y> \gamma_1(Y|\mathcal{I}_{\mathbb{Q}}) -a|\mathcal{I}_{\mathbb{Q}}) \quad \mathbb{Q}-\textrm{a.s.},
\end{align*}
where the last equality is a~consequence of~\eqref{zalozenia3} and Lemma \ref{bez.war}.

Similarly, letting $m\to\infty$ in \eqref{r2} yields
\begin{align*}
\limsup_{n\to\infty}
\ K^\mathbb{Y}(k_n,n,a)/n & \le
\mathbb{E}_{\mathbb{Q}}(I(\gamma_1(Y|\mathcal{I}_{\mathbb{Q}})-a\leq Y\leq \gamma_1(Y|\mathcal{I}_{\mathbb{Q}}))|\mathcal{I}_{\mathbb{Q}})\\
&=\mathbb{Q}(\gamma_1(Y|\mathcal{I}_{\mathbb{Q}})-a\leq Y\leq \gamma_1(Y|\mathcal{I}_{\mathbb{Q}})|\mathcal{I}_{\mathbb{Q}} ) \\
&=\mathbb{Q}(Y\geq  \gamma_1(Y|\mathcal{I}_{\mathbb{Q}}) -a|\mathcal{I}_{\mathbb{Q}})\\
&=\mathbb{Q}(Y>  \gamma_1(Y|\mathcal{I}_{\mathbb{Q}}) -a|\mathcal{I}_{\mathbb{Q}}) \quad \mathbb{Q}-\textrm{a.s.},
\end{align*}
by \eqref{zalozenia2}  and Lemma \ref{bez.war}.

Summarizing, we have proved  that 
\begin{align*}
\liminf_{n\to\infty}K^\mathbb{Y}(k_n,n,a)/n &= \limsup_{n\to\infty}K^\mathbb{Y}(k_n,n,a)/n \\
& =\mathbb{Q}(Y>  \gamma_1(Y|\mathcal{I}_{\mathbb{Q}}) -a|\mathcal{I}_{\mathbb{Q}}) \quad \mathbb{Q}-\textrm{a.s.},
\end{align*}
which gives \eqref{teza.tw.3}. The  proof is complete.
\end{proof}

\medskip
We next show  that Theorem~\ref{tw.3} remains true when the sigma-field $\mathcal{I}_{\mathbb{Q}}$ is replaced by~$\tilde{\mathcal{I}}$. 
\begin{tw}\label{tw.3_2}
Under the assumptions of~Theorem~\ref{tw.1} and provided that
\begin{align}
    \label{zalozenia6}
    \mathbb{Q}(Y=\gamma_1(Y|\tilde{\mathcal{I}}))=0,\\
    \label{zalozenia5}
    \mathbb{Q}(Y=\gamma_1(Y|\tilde{\mathcal{I}})-a)=0,
\end{align}
we have
\begin{equation}\label{teza_tw3}
K^\mathbb{Y}(k_n,n,a)/n\xrightarrow{\mathbb{Q}-a.s.}\mathbb{Q}(Y>\gamma_1(Y|\tilde{\mathcal{I}})-a|\tilde{\mathcal{I}})\textrm{ as }n\to\infty.
\end{equation}
\end{tw}
\begin{proof}
Analysis similar to that in the proof of Theorem~\ref{tw.3} shows the conclusion. Indeed, by Theorem  \ref{tw.statystyki2.4}, \eqref{przed3.5} can be replaced by \eqref{teza3}. Therefore defining $L_m$ and $R_m$ we use $\tilde{\mathcal{I}}$ rather than  $\mathcal{I}_{\mathbb{Q}}$. Next, the first  part of Lemma~\ref{l2} (ii) guarantees that an analogue of \eqref{po3.6} holds for so defined $L_m$ and $R_m$. Finally,  the ergodic theorem, used in~\eqref{ergodic_th}, is also satisfied with sigma-field~$\tilde{\mathcal{I}}$; see, for example, Grimmett and Stirzaker (2004, Chapter 9). The rest of the proof runs as before with $\mathcal{I}_{\mathbb{Q}}$ replaced by $\tilde{\mathcal{I}}$.
\end{proof}

\begin{rem}
\label{rem3.1}
 Observe that Theorems \ref{tw.1} and \ref{tw.statystyki2.4} give 
\begin{equation}
\label{rowne.konceI}
\gamma_1(Y|\mathcal{I}_{\mathbb{Q}})=\gamma_1(Y|\tilde{\mathcal{I}}) \quad \mathbb{Q}-\textrm{a.s.}
\end{equation}
Therefore  it does not matter if in Theorem \ref{tw.3_2} (and Theorem \ref{tw.3}) we assume \eqref{zalozenia6} and \eqref{zalozenia5} or \eqref{zalozenia3}  and \eqref{zalozenia2}. 

\end{rem}

\medskip

Theorems~\ref{tw.3} and \ref{tw.3_2}  deal with the case of strictly stationary sequences of rv's defined on the probability space $(\mathbb{R}^{\mathbb{N}},\mathcal{B}(\mathbb{R}^{\mathbb{N}}),\mathbb{Q})$. The reminder of this section will be devoted  to  deriving a~more general results that holds  for strictly stationary sequences  of rv's  existing in any probability triple $(\Omega,\mathcal{F},\mathbb{P})$. We begin with  a~lemma which is a  fairly straightforward extension of Theorem~\ref{tw.3}. This lemma will be a starting point in obtaining our desired generalization.

\begin{lem}
\label{lemY}
Let $\mathbb{X}=(X_n,n\ge 1)$ be a~strictly stationary sequence of rv's defined on the probability space $(\Omega,\mathcal{F},\mathbb{P})$, $\mathbb{Q}$ be a~stationary measure on $(\mathbb{R}^{\mathbb{N}},\mathcal{B}(\mathbb{R}^{\mathbb{N}}))$ given by 
\begin{equation}\label{miara}
\mathbb{Q}(B)=\mathbb{P}(\mathbb{X}\in B)\quad\textrm{for all }B\in\mathcal{B}(\mathbb{R}^{\mathbb{N}}),
\end{equation}
and $Y$ be an rv on $(\mathbb{R}^{\mathbb{N}},\mathcal{B}(\mathbb{R}^{\mathbb{N}}),\mathbb{Q})$ such that 
\begin{equation}\label{def_Y2}
Y((x_1,x_2,\ldots))=x_1\quad\textrm{for }(x_1,x_2,\ldots)\in\mathbb{R}^{\mathbb{N}}.
\end{equation}
If \eqref{zalozenia3} and \eqref{zalozenia2} hold and $(k_n,n\ge 1)$ satisfies~\eqref{warK}, then there exists an rv $W$ on $(\Omega,\mathcal{F},\mathbb{P})$ for which we have
\[
K^\mathbb{X}(k_n,n,a)/n\xrightarrow{\mathbb{P}-a.s.}W\textrm{ as }n\to\infty.
\]
Moreover, the joint distribution of  $W$ and $(K^X(k_n,n,a)/n,n\ge 1)$ is the same as the joint distribution of $\mathbb{Q}(Y> \gamma_1(Y|\mathcal{I}_{\mathbb{Q}}) -a|\mathcal{I}_{\mathbb{Q}})$ and $(K^Y(k_n,n,a)/n,n\ge 1)$, where the sequence $\mathbb{Y}=(Y_n,n\ge 1)$ is defined by~\eqref{def_Y}.
\end{lem}

\begin{proof}
We first observe that Theorem \ref{tw.3} gives \eqref{teza.tw.3}. Next we apply a method used in the proof of  Theorem~6 in~Dembi\'nska (2017). More precisely, we notice that \eqref{miara}, \eqref{def_Y2} and \eqref{def_Y}  guarantee  that the sequences $(X_n,n\ge 1)$ and $(Y_n,n\ge 1)$ have the same distribution. Since the almost sure convergence of a sequence of rv's to some rv is determined by the joint distribution of this sequence and the limiting rv, \eqref{teza.tw.3} yields the assertion of the lemma.
\end{proof}

Although Lemma~\ref{lemY} concerns the asymptotic behavior of $K^\mathbb{X}(k_n,n,a)$, that is of  a  counting  rv based on the sequence  $\mathbb{X}=(X_n,n\ge 1)$,  assumptions \eqref{zalozenia3} and \eqref{zalozenia2} impose restrictions on the auxiliary  sequence~$\mathbb{Y}=(Y_n,n\ge 1)$. This auxiliary  sequence~$\mathbb{Y}$ is also used to describe the distribution of the limiting rv $W$.

It is desirable to establish a reformulation of Lemma~\ref{lemY} in which the sequence $\mathbb{Y}$ does not appear so that the assumptions and conclusion are expressed only in terms of $\mathbb{X}=(X_n,n\ge 1)$. We first get rid of $\mathbb{Y}$ from assumptions  \eqref{zalozenia3} and \eqref{zalozenia2}.
\begin{lem}
\label{lem.zal.bez.Y}
Assumptions \eqref{zalozenia3} and \eqref{zalozenia2} of Lemma~\ref{lemY} are equivalent to
\begin{equation}
    \mathbb{P}(X_1=\gamma_1(X_1|\mathcal{I}^{\mathbb{X}}))=0,    
\label{zal_x3} 
\end{equation}
and
\begin{equation}
    \mathbb{P}(X_1=\gamma_1(X_1|\mathcal{I}^{\mathbb{X}})-a)=0,
\label{zal_x2}
\end{equation}
respectively.
\end{lem}
\begin{proof}
We will restrict ourselves to showing that \eqref{zalozenia3} and \eqref{zal_x3} are equivalent. The proof of the equivalence of  \eqref{zalozenia2} and \eqref{zal_x2} goes along the same lines.

By Theorem \ref{tw.1}, \eqref{zalozenia3} can be rewritten as
\begin{equation}
\label{LEQ1e1}
\mathbb{Q}(Y=  \lim_{n\to\infty}Y_{n:n})=0,
\end{equation}
where $\lim_{n\to\infty} Y_{n:n}$ is defined to equal, for example, $+\infty$ on the set (of probability $\mathbb{Q}$ zero) of $(x_1,x_2,\ldots)\in\mathbb{R}^{\mathbb{N}}$ such that $\lim_{n\to\infty} Y_{n:n}((x_1,x_2,\ldots))$ does not exist.

On the other hand,  by Theorem \ref{tw.2}, \eqref{zal_x3}  is equivalent to the following condition
\begin{equation}
\label{LEQ1e2}
\mathbb{P}(X_1=  \lim_{n\to\infty}X_{n:n})=0,
\end{equation}
where we define $\lim_{n\to\infty} X_{n:n}$  to be  equal to, for example, $+\infty$ on the set (of probability $\mathbb{P}$ zero) of $\omega\in\Omega$ such that $(X_{n:n}(\omega), n\geq 1)$ does not have a limit.

Hence, to show the equivalence of \eqref{zalozenia3} and \eqref{zal_x3}, it suffices to prove that \eqref{LEQ1e1} and \eqref{LEQ1e2} are equivalent. But  \eqref{miara} and \eqref{def_Y2} yield
$$
\mathbb{Q}(Y=  \lim_{n\to\infty}Y_{n:n})=\mathbb{P}(X_1=  \lim_{n\to\infty}X_{n:n}),
$$
which clearly forces the required equivalence.
\end{proof}

What is still lacking is an explicit description of the limiting rv $W$, without using  the auxiliary  sequence~$\mathbb{Y}=(Y_n,n\ge 1)$. We provide it in the following theorem and thus give the desired generalization of Theorem~\ref{tw.3} and the main result of this section.

\begin{tw}\label{tw.5}
Let $\mathbb{X}=(X_n,n\ge 1)$ be a~strictly stationary sequence of rv's and $(k_n,n\ge 1)$ be a~sequence satisfying~\eqref{warK}. Assume that conditions \eqref{zal_x3} and  \eqref{zal_x2} hold. Then
\begin{equation}\label{r7}
K^\mathbb{X}(k_n,n,a)/n\xrightarrow{\mathbb{P}-a.s.}\mathbb{P}(X_1> \gamma_1(X_1|\mathcal{I}^{\mathbb{X}}) -a|\mathcal{I}^{\mathbb{X}})\textrm{ as }n\to\infty.
\end{equation}
\end{tw}

\begin{proof}
By Lemma~\ref{lem.zal.bez.Y}, assumptions \eqref{zalozenia3} and \eqref{zalozenia2} of Theorem~\ref{tw.3} are satisfied with the measure $\mathbb{Q}$ described by~\eqref{miara}. Using this theorem we obtain that~\eqref{teza.tw.3} holds with the sequence $\mathbb{Y}=(Y_n, n\ge 1)$ defined by \eqref{def_Y}, where the rv $Y$ is given by~\eqref{def_Y2}. From \eqref{teza.tw.3} we see that $\lim_{n\to\infty} K^{\mathbb{Y}}(k_n,n,a)/n$ exists $\mathbb{Q}$-a.s. On the set of $\mathbb{Q}$-probability zero of $(x_1,x_2,\ldots)\in\mathbb{R}^{\mathbb{N}}$ for which $\lim_{n\to\infty} K^{\mathbb{Y}}(k_n,n,a)((x_1,x_2,\ldots))/n$ does not exist we can define, for example, $\lim_{n\to\infty} K^{\mathbb{Y}}(k_n,n,a)((x_1,x_2,\ldots))/n=0$. Moreover \eqref{teza.tw.3} ensures that
\begin{equation}\label{d33w1}
    \lim_{n\to\infty} K^{\mathbb{Y}}(k_n,n,a)/n\textrm{ is }\mathcal{I}_{\mathbb{Q}}-\textrm{measurable}
\end{equation}
and
\begin{equation}\label{d33w2}
    \mathbb{E}_{\mathbb{Q}}(\lim_{n\to\infty} K^{\mathbb{Y}}(k_n,n,a)/n\cdot I(\tilde{G}))
    = \mathbb{E}_{\mathbb{Q}} (I(Y>\gamma_1(Y|\mathcal{I}_{\mathbb{Q}})-a)\cdot I(\tilde{G}))\textrm{ for all }\tilde{G}\in\mathcal{I}_{\mathbb{Q}}.
\end{equation}

On the other hand, Lemma~\ref{lemY} says that $\lim_{n\to\infty} K^{\mathbb{X}}(k_n,n,a)/n$ exists $\mathbb{P}$-a.s. For definiteness, we can assume that, for instance, $\lim_{n\to\infty} K^{\mathbb{X}}(k_n,n,a)/n (\omega)=0$ for $\omega\in\Omega$ such that this limit does not exist. To prove~\eqref{r7}, we must show that
\begin{equation}\label{d33w3}
    \lim_{n\to\infty} K^{\mathbb{X}}(k_n,n,a)/n\textrm{ is }\mathcal{I}^{\mathbb{X}}-\textrm{measurable} 
\end{equation}
and
\begin{equation}\label{d33w4}
    \mathbb{E}_{\mathbb{P}}(\lim_{n\to\infty} K^{\mathbb{X}}(k_n,n,a)/n\cdot I(G))
    = \mathbb{E}_{\mathbb{P}} (I(X_1>\gamma_1(X_1|\mathcal{I}^{\mathbb{X}})-a)\cdot I(G))\textrm{ for all }G\in\mathcal{I}^{\mathbb{X}}.
\end{equation}
Condition \eqref{d33w3} means that, for every $A\in\mathcal{B}(\mathbb{R})$,
\[
\{\omega\in\Omega\colon\lim_{n\to\infty} \frac{K^{\mathbb{X}}(k_n,n,a)}{n}(\omega)\in A\}\in\mathcal{I}^{\mathbb{X}},
\]
where $\mathcal{B}(\mathbb{R})$ stands for the Borel sigma-field of subsets of~$\mathbb{R}$. By the definition of $\mathcal{I}^{\mathbb{X}}$, we can rewrite the above requirement as
\begin{align}\label{d33w5}
\textrm{for every }A\in \mathcal{B}(\mathbb{R})\textrm{ there exists }B\in\mathcal{B}(\mathbb{R}^{\mathbb{N}})\textrm{ such that, for any }i\ge 1,\nonumber\\
\{\omega\in\Omega\colon \lim_{n\to\infty} \frac{K^\mathbb{X}(k_n,n,a)}{n}(\omega)\in A\}=
\{\omega\in\Omega\colon (X_i(\omega),X_{i+1}(\omega),\ldots)\in B\}.
\end{align}
To show \eqref{d33w5}, fix $A\in\mathcal{B}(\mathbb{R})$ and observe that we can take
\begin{align}\label{d33w6}
B=\{(x_1,x_2,\ldots)\in\mathbb{R}^{\mathbb{N}}\colon 
\lim_{n\to\infty} \frac{K^\mathbb{X}(k_n,n,a)}{n}((x_1,x_2,\ldots)) \in A\}\nonumber\\
=\{(x_1,x_2,\ldots)\in\mathbb{R}^{\mathbb{N}}:\ \lim_{n\to\infty} \sum_{i=1}^n \frac{x_i\in(x_{k_n:n}-a,x_{k_n:n})}{n}\in A\}. 
\end{align}
Indeed, by Theorem~\ref{tw.3_2} and Remark~\ref{rem3.1}, $B\in\tilde{\mathcal{I}}$ so in consequence $B\in\mathcal{B}(\mathbb{R}^{\mathbb{N}})$ and
\[
B=\{(x_1,x_2,\ldots)\in\mathbb{R}^{\mathbb{N}}\colon 
(x_i,x_{i+1},\ldots) \in B\}\textrm{ for all }i\ge 1,
\]
which gives, for all $i\ge 1$,
\begin{align*}
\{\omega\in\Omega&\colon (X_i(\omega),X_{i+1}(\omega),\ldots)\in B\}=\{\omega\in\Omega\colon (X_1(\omega),X_{2}(\omega),\ldots)\in B\}\\
&= \{\omega\in\Omega\colon \lim_{n\to\infty} \sum_{i=1}^n \frac{X_i(\omega)\in (X_{k_n:n}(\omega)-a,X_{k_n:n}(\omega))}{n}\in A\}\\
&= \{\omega\in\Omega\colon \lim_{n\to\infty} \frac{K^\mathbb{X}(k_n,n,a)}{n}(\omega)\in A\},
\end{align*}
where the second equality follows from~\eqref{d33w6}. Thus~\eqref{d33w5} holds as required and the proof of~\eqref{d33w3} is complete.

We next prove~\eqref{d33w4}. For this purpose, fix $G\in\mathcal{I}^{\mathbb{X}}$. By Lemma~\ref{l2}~(iii), there exists a~set $\tilde{G}\in\mathcal{I}_{\mathbb{Q}}$ such that
\begin{equation}\label{d33w7}
    G=\{\omega\in\Omega\colon (X_i(\omega),X_{i+1}(\omega),\ldots)\in \tilde{G}\}\textrm{ for all }i\ge 1.
\end{equation}
Moreover, for $G$ and $\tilde{G}$ satisfying~\eqref{d33w7}, 
\begin{equation}\label{d33w8}
(\mathbb{X},I(G))\textrm{ and }(\mathbb{Y},I(\tilde{G}))\textrm{ have the same distribution.}
\end{equation}
To show this, it suffices to observe that, for all $C\in\mathcal{B}(\mathbb{R}^{\mathbb{N}})$, we have
\[
\mathbb{P}(\mathbb{X}\in C, I(G)=1)=\mathbb{Q}(\mathbb{Y}\in C, I(\tilde{G})=1).
\]
But
\begin{align*}
\mathbb{P}(\mathbb{X}\in C, I(G)=1)&= \mathbb{P}(\{\mathbb{X}\in C\}\cap G)
= \mathbb{P}(\{\mathbb{X}\in C\}\cap \{\mathbb{X}\in\tilde{G}\})\\ 
&= \mathbb{Q}(C\cap\tilde{G})
= \mathbb{Q}(\mathbb{Y}\in C, I(\tilde{G})=1)
\end{align*}
as required, the second and third equality being consequences of~\eqref{d33w7} and~\eqref{miara}, respectively.

Theorem~\ref{tw.2} gives
\begin{equation}\label{d33w9}
    \mathbb{E}_{\mathbb{P}} (I(X_1>\gamma_1(X_1|\mathcal{I}^{\mathbb{X}})-a)\cdot I(G))
    = \mathbb{E}_{\mathbb{P}} (I(X_1>\lim_{n\to\infty} X_{k_n:n} -a)\cdot I(G)),
\end{equation}
where, analogously as in the proof of Lemma~\ref{lem.zal.bez.Y}, we define $\lim_{n\to\infty} X_{k_n:n}$ to be equal to, for example, $+\infty$ on the set (of probability $\mathbb{P}$ zero) of $\omega\in\Omega$ such that $(X_{k_n:n}(\omega),n\ge 1)$ does not have a~limit. Applying~\eqref{d33w8} and next Theorem~\ref{tw.1}, \eqref{d33w2} and~\eqref{d33w8} we obtain
\begin{align}\label{d33w10}
    \mathbb{E}_{\mathbb{P}} (I(X_1&>\lim_{n\to\infty} X_{k_n:n} -a)\cdot I(G))
    = \mathbb{E}_{\mathbb{Q}} (I(Y>\lim_{n\to\infty} Y_{k_n:n} -a)\cdot I(\tilde{G}))\nonumber\\
    &= \mathbb{E}_{\mathbb{Q}} (I(Y>\gamma_1(Y|\mathcal{I}_{\mathbb{Q}}) -a)\cdot I(\tilde{G}))
    = \mathbb{E}_{\mathbb{Q}} \Big(\lim_{n\to\infty} \frac{K^\mathbb{Y}(k_n,n,a)}{n} \cdot I(\tilde{G})\Big)\nonumber\\
    &= \mathbb{E}_{\mathbb{P}} \Big(\lim_{n\to\infty} \frac{K^\mathbb{X}(k_n,n,a)}{n} \cdot I(G)\Big),
\end{align}
where again, on the set (of probability $\mathbb{Q}$ zero) of $(x_1,x_2,\ldots)\in\mathbb{R}^{\mathbb{N}}$ such that 
$\lim_{n\to\infty} Y_{k_n:n}((x_1,x_2,\ldots))$ does not exist, this limit is defined to equal, for example, $+\infty$. Since $G\in\mathcal{I}^{\mathbb{X}}$ is arbitrary, \eqref{d33w9} and \eqref{d33w10} show that \eqref{d33w4} holds. This completes the proof.
\end{proof}

Theorem~\ref{tw.5} can be reformulated as follows.

\begin{tw}\label{tw.3.4}
Under the assumptions of Theorem~\ref{tw.5},
\[
K^\mathbb{X}(k_n,n,a)/n\xrightarrow{\mathbb{P}-a.s.} I(\gamma_1(X_1|\mathcal{I}^{\mathbb{X}})<\infty)\cdot \mathbb{P}(X_1>\gamma_1(X_1|\mathcal{I}^{\mathbb{X}}) -a|\mathcal{I}^{\mathbb{X}})\textrm{ as }n\to\infty.
\]
\end{tw}

\begin{proof}
Observe that $\mathbb{P}-\textrm{a.s.}$
\begin{align*}
&\mathbb{P}(X_1>\gamma_1(X_1|\mathcal{I}^{\mathbb{X}})-a|\mathcal{I}^{\mathbb{X}})=\mathbb{E}_{\mathbb{P}}\Big(I(X_1>\gamma_1(X_1|\mathcal{I}^{\mathbb{X}})-a)|\mathcal{I}^{\mathbb{X}}\Big)\\
&= \mathbb{E}_{\mathbb{P}}\Big(I(\gamma_1(X_1|\mathcal{I}^{\mathbb{X}})=\infty)\cdot
I(X_1>\gamma_1(X_1|\mathcal{I}^{\mathbb{X}})-a)|\mathcal{I}^{\mathbb{X}}\Big)\\
&+ \mathbb{E}_{\mathbb{P}}\Big(I(\gamma_1(X_1|\mathcal{I}^{\mathbb{X}})<\infty)\cdot
I(X_1>\gamma_1(X_1|\mathcal{I}^{\mathbb{X}})-a)|\mathcal{I}^{\mathbb{X}}\Big)\\
&=0+I(\gamma_1(X_1|\mathcal{I}^{\mathbb{X}})<\infty)\cdot
\mathbb{E}_{\mathbb{P}}(I(X_1>\gamma_1(X_1|\mathcal{I}^{\mathbb{X}})-a)|\mathcal{I}^{\mathbb{X}})\\
&=I(\gamma_1(X_1|\mathcal{I}^{\mathbb{X}})<\infty)\cdot
\mathbb{P}(X_1>\gamma_1(X_1|\mathcal{I}^{\mathbb{X}})-a|\mathcal{I}^{\mathbb{X}}),
\end{align*}
where the next to last equality is due to the facts that
\[
I(\gamma_1(X_1|\mathcal{I}^{\mathbb{X}})=\infty)\cdot
I(X_1>\gamma_1(X_1|\mathcal{I}^{\mathbb{X}})-a)=0\quad\mathbb{P}-\textrm{a.s.},
\]
and $I(\gamma_1(X_1|\mathcal{I}^{\mathbb{X}})<\infty)$ is $\mathcal{I}^{\mathbb{X}}$-measurable.

Now it is obvious that Theorem~\ref{tw.3.4} is a~consequence of Theorem~\ref{tw.5}.
\end{proof}

In particular, Theorem~\ref{tw.3.4} shows that, under assumptions of this theorem, for $\mathbb{P}$-almost every
$\omega\in\Omega$ such that $\gamma_1(X_1|\mathcal{I}^{\mathbb{X}})(\omega)=\infty$ (where $\gamma_1(X_1|\mathcal{I}^{\mathbb{X}})$ is any version of conditional right endpoint of the support of $X_1$ given $\mathcal{I}^{\mathbb{X}}$), we have
\[
K^\mathbb{X}(k_n,n,a)(\omega)/n\longrightarrow 0\textrm{ as }n\to\infty.
\]
Moreover, from Theorem~\ref{tw.3.4} we can easily deduce the following corollary.


\begin{co}\label{con3.1}
If $(k_n,n\ge 1)$ satisfies~\eqref{warK} and $\mathbb{X}=(X_n,n\ge 1)$ is a~strictly stationary sequence of rv's such that $\gamma_1(X_1|\mathcal{I}^{\mathbb{X}})=\infty$ $\mathbb{P}$-a.s., then
\begin{equation}
K^\mathbb{X}(k_n,n,a)/n\xrightarrow{\mathbb{P}-a.s.}0\textrm{ as }n\to\infty.
\label{zbiega_0}
\end{equation}
\end{co}

\section{Extensions to random sample sizes}
\label{Sec4}
The aim of this section is to provide results yielding information about the long-term behavior of the model of insurance claims described in the Introduction. First we focus on the asymptotic behavior of normalized numbers of near maximum insurance claims that occur up to time~$t$, ${\cal K}_t(a)/c(t)$, where $c(\cdot)$ is some positive function. This asymptotic behavior will not change if we replace ${\cal K}_t(a)$ by ${\cal K}_t(k_{N(t)},a)$, where $(k_n,n\ge 1)$ is a~sequence satisfying~\eqref{warK}. Therefore we give limiting theorems not only for ${\cal K}_t(a)/c(t)$, but more generally for ${\cal K}_t(k_{N(t)},a)/c(t)$ - normalized numbers of insurance claims that are  in a left neighborhood of  the $(N(t)-k_{N(t)}+1)$st largest claim, where $(k_n,n\ge 1)$ satisfies~\eqref{warK}.

We will derive the limiting theorems of interest from the corresponding results for $K(k_n,n,a)/n$. For this purpose we will use two lemmas. The first lemma is taken from Embrechts et al. (1997, Lemma 2.5.3).
\begin{lem}\label{LEKM}
Let $W,W_1,W_2,\ldots$ be rv's satisfying $W_n\xrightarrow{a.s.}W$ as $n\to\infty$ and $(N(t),t\ge 0)$ be a~process of non-negative integer-valued rv's.
\begin{description}
 \item[(i)]
 If $N(t)\xrightarrow{a.s.}\infty$ as $t\to\infty$, then
 $W_{N(t)}\xrightarrow{a.s.}W$ as $n\to\infty$.
 \item[(ii)]
 If $N(t)\xrightarrow{p}\infty$ as $t\to\infty$, then
 $W_{N(t)}\xrightarrow{p}W$ as $n\to\infty$.
\end{description}
\end{lem}

Lemma~\ref{LEKM} and Theorem~\ref{tw.3.4} immediately imply the following result.
\begin{tw}~\label{th4.1}
Under the assumptions of Theorem~\ref{tw.5}, if moreover $(N(t),t\ge 0)$ is a~process of non-negative integer-valued rv's $N(t)$ such that $N(t)\xrightarrow{a.s.}\infty$ ($N(t)\xrightarrow{p}\infty$) we have, as $t\to\infty$,
\begin{align}\label{convas}
   {\cal K}_t(k_{N(t)},a)/N(t)\xrightarrow{\mathbb{P}-a.s.} I(\gamma_1(X_1|\mathcal{I}^{\mathbb{X}})<\infty)\cdot \mathbb{P}(X_1>\gamma_1(X_1|\mathcal{I}^{\mathbb{X}}) -a|\mathcal{I}^{\mathbb{X}})\\
    \label{convP}
   \big({\cal K}_t(k_{N(t)},a)/N(t)\xrightarrow{p} I(\gamma_1(X_1|\mathcal{I}^{\mathbb{X}})<\infty)\cdot \mathbb{P}(X_1>\gamma_1(X_1|\mathcal{I}^{\mathbb{X}}) -a|\mathcal{I}^{\mathbb{X}})\big).
\end{align}
\end{tw}
It is worth pointing out that the condition $N(t)\xrightarrow{a.s.}\infty$ is satisfied whenever $(N(t),t\ge 0)$ is a~renewal counting process; see Embrechts et al. (1997, Section 2.5.2).

If in Theorem~\ref{th4.1} we want to replace the random normalizing process $N(t)$ by some deterministic function of~$t$, we need to add an assumption concerning the order of magnitude of~$N(t)$ to~$\infty$. For insurance and finance applications, it often suffices to restrict attention to process $(N(t),t\ge 0)$ satisfying $N(t)/c(t)\xrightarrow{a.s.}Z$, where $c(\cdot)$ is some positive function such that $c(t)\to\infty$ and $Z$ is some almost surely positive rv. The condition $N(t)/c(t)\xrightarrow{a.s.}Z$ is fulfilled with $c(t)=t$ and $Z=\lambda$, where $\lambda$ is a~positive constant, whenever $(N(t),t\ge 0)$ is a~renewal counting process with claim times having finite expectations; see Embrechts et al. (1997, Theorem 2.5.10). In particular it is satisfied when $(N(t),t\ge 0)$ is a~homogeneous Poisson process with intensity~$\lambda$.

We introduce the following condition to shorten formulation of results that are given in the sequel.
\begin{cond}\label{cond_A}
$(N(t),t\ge 0)$ is a~process of non-negative integer-valued rv's, $Z$ is an almost surely positive rv and $c(\cdot)$ is a positive function satisfying $c(t)\to\infty$ as $t\to\infty$.
\end{cond}

\begin{lem}\label{4lemat2}
Under Condition~\ref{cond_A}, we have
\begin{description}
 \item[(i)]
 if $N(t)/c(t)\xrightarrow{a.s.}Z$, then $N(t)\xrightarrow{a.s.}\infty$;
 \item[(ii)]
 if $N(t)/c(t)\xrightarrow{p}Z$, then $N(t)\xrightarrow{p}\infty$;
 \item[(iii)]
 if $N(t)/c(t)\xrightarrow{d}Z$, then $N(t)\xrightarrow{p}\infty$;
\end{description}
\end{lem}
\begin{proof}
Part (i) is immediate, because $\lim_{n\to\infty} N(t)(\omega)=\infty$ for every $\omega\in\Omega$, where $A=\{\omega\in\Omega\colon \lim_{n\to\infty}\frac{N(t)(\omega)}{c(t)}=Z(\omega)\textrm{ and }Z(\omega)>0\}$ and $\mathbb{P}(A)=1$ by assumption.

Part (ii) follows from (iii), because the convergence in probability implies that in distribution.

What is left is to show (iii).  To do this, fix $D>0$ and choose a~sequence $(t_n,n\ge 1)$ with the property that for $t\ge t_n$, $c(t)\ge n$.
The existence of such a~sequence is guaranteed by the assumption that $c(t)\to\infty$ as $t\to\infty$. Then, for $t\ge t_n$, and provided that
\begin{equation}\label{gwiazda}
    D/n\textrm{ is a continuity point of the cumulative distribution function of }Z,
\end{equation}
we have
\[
\mathbb{P}(N(t)\le D)= \mathbb{P}\Big(\frac{N(t)}{c(t)}\le \frac{D}{c(t)}\Big)\le \mathbb{P}\Big(\frac{N(t)}{c(t)}\le \frac{D}{n}\Big)\to \mathbb{P}\Big(Z\le\frac{D}{n}\Big)\textrm{ as }t\to\infty,
\]
where the convergence is due to the assumption that $N(t)/c(t)\xrightarrow{d}Z$ as $t\to\infty$. Since $D$ is an arbitrary positive constant satisfying~\eqref{gwiazda}, we obtain, for all but at most countably many~$D$,
\begin{equation}\label{s4lw1}
    \limsup_{t\to\infty} \mathbb{P}(N(t)\le D)\le \mathbb{P}\Big(Z\le\frac{D}{n}\Big).
\end{equation}
By letting $n\to\infty$ and using the right continuity of cumulative distribution functions we see that~\eqref{s4lw1}
implies $\limsup_{t\to\infty} \mathbb{P}(N(t)\le D)\le 0$.
It follows that
\[
\limsup_{t\to\infty} \mathbb{P}(N(t)\le D)= 0\textrm{ for any }D>0,
\]
and this is precisely the conclusion of~(iii).
\end{proof}

Combining Theorem~\ref{th4.1} and Lemma~\ref{4lemat2} we immediately obtain the following result.

\begin{tw}\label{th4.2}
Let Condition~\ref{cond_A} and the assumptions of Theorem~\ref{tw.5} hold.
\begin{description}
 \item[(i)]
 If
 \begin{equation}\label{as1th4.2}
     N(t)/c(t)\xrightarrow{a.s.}Z\quad \left(N(t)/c(t)\xrightarrow{p}Z\right)\textrm{ as }t\to\infty,
 \end{equation}
 then, as $t\to\infty$,
 \begin{align}\label{c2th4.2}
     {\cal K}_t(k_{N(t)},a)/c(t)\xrightarrow{a.s.} Z\cdot I(\gamma_1(X_1|\mathcal{I}^{\mathbb{X}})<\infty)\cdot \mathbb{P}(X_1>\gamma_1(X_1|\mathcal{I}^{\mathbb{X}}) -a|\mathcal{I}^{\mathbb{X}})\nonumber\\
     \left({\cal K}_t(k_{N(t)},a)/c(t)\xrightarrow{p} Z\cdot I(\gamma_1(X_1|\mathcal{I}^{\mathbb{X}})<\infty)\cdot \mathbb{P}(X_1>\gamma_1(X_1|\mathcal{I}^{\mathbb{X}}) -a|\mathcal{I}^{\mathbb{X}})\right).
 \end{align}
 \item[(ii)]
 If
 \begin{equation}\label{as2th4.2}
     N(t)/c(t)\xrightarrow{d}Z\textrm{ as }t\to\infty,
 \end{equation}
 then, as $t\to\infty$,
  \begin{align}\label{c3th4.2}
     {\cal K}_t(k_{N(t)},a)/c(t)\xrightarrow{d} Z\cdot I(\gamma_1(X_1|\mathcal{I}^{\mathbb{X}})<\infty)\cdot \mathbb{P}(X_1>\gamma_1(X_1|\mathcal{I}^{\mathbb{X}}) -a|\mathcal{I}^{\mathbb{X}})
 \end{align}
 provided that $I(\gamma_1(X_1|\mathcal{I}^{\mathbb{X}})<\infty)\cdot \mathbb{P}(X_1>\gamma_1(X_1|\mathcal{I}^{\mathbb{X}}) -a|\mathcal{I}^{\mathbb{X}})$ is degenerate. If in turn $Z$ is degenerate, then~\eqref{as2th4.2} implies $N(t)/c(t)\xrightarrow{p}Z$ and hence~\eqref{c2th4.2} holds by~(i).
\end{description}
\end{tw}

\begin{proof}
Clearly
\begin{equation}\label{pth4.1.0}
    \frac{{\cal K}_t(k_{N(t)},a)}{c(t)}= \frac{{\cal K}_t(k_{N(t)},a)}{N(t)}\cdot \frac{N(t)}{c(t)}.
\end{equation}
To prove~(i) we use Lemma~\ref{4lemat2} and Theorem~\ref{th4.1} to observe that~\eqref{convas} (\eqref{convP}) holds. From this and~\eqref{as1th4.2} we obtain
\begin{align}
\Big(\frac{{\cal K}_t(k_{N(t)},a)}{N(t)}, \frac{N(t)}{c(t)}\Big)\xrightarrow{a.s.}  \Big( I(\gamma_1(X_1|\mathcal{I}^{\mathbb{X}})<\infty)\cdot \mathbb{P}(X_1>\gamma_1(X_1|\mathcal{I}^{\mathbb{X}}) -a|\mathcal{I}^{\mathbb{X}}), Z\Big)  \label{pth4.1.1} \\
\left(  \Big(\frac{{\cal K}_t(k_{N(t)},a)}{N(t)}, \frac{N(t)}{c(t)}\Big)\xrightarrow{p}  \Big( I(\gamma_1(X_1|\mathcal{I}^{\mathbb{X}})<\infty)\cdot \mathbb{P}(X_1>\gamma_1(X_1|\mathcal{I}^{\mathbb{X}}) -a|\mathcal{I}^{\mathbb{X}}), Z\Big)\right).    \label{pth4.1.2}
 \end{align}
The above joint convergence is obvious in the a.s. version. Its validity in the case of convergence in probability follows from, for example, Theorem~2.7~(vi) of~van der Vaart (1998).
Combining~\eqref{pth4.1.0} with~\eqref{pth4.1.1} and~\eqref{pth4.1.2}, and using Continuous Mapping Theorem finish the proof of~(i).

For (ii) observe that part (iii) of Lemma~\ref{4lemat2} and Theorem~\ref{th4.1} imply~\eqref{convP}. Moreover, the limiting rv in~\eqref{convP} is degenerate by assumption. Applying~\eqref{pth4.1.0}, \eqref{convP}, \eqref{as2th4.2} and Slutsky Lemma gives~\eqref{c3th4.2}, and the proof is complete.
\end{proof}

In the case when  $\gamma_1(X_1|\mathcal{I}^{\mathbb{X}})=\infty$ a.s. Theorem~\ref{th4.2} takes on a~simpler form.
\begin{co}\label{c4.1}
Let Condition~\ref{cond_A} hold, $\mathbb{X}=(X_n,n\ge 1)$ be a~strictly stationary sequence of rv's such that $\gamma_1(X_1|\mathcal{I}^{\mathbb{X}})=\infty$ a.s. and $(k_n,n\ge 1)$ be a~sequence satisfying~\eqref{warK}. If
\begin{equation}\label{ac4.1}
    N(t)/c(t)\xrightarrow{a.s.}Z\quad (N(t)/c(t)\xrightarrow{d}Z)\textrm{ as }t\to\infty,
\end{equation}
then
\begin{equation}\label{tc4.1}
    \frac{{\cal K}_t(k_{N(t)},a)}{c(t)}\xrightarrow{a.s.}0\quad
   \left(\frac{{\cal K}_t(k_{N(t)},a)}{c(t)}\xrightarrow{p}0\right) \textrm{ as }t\to\infty.
\end{equation}
\end{co}
\begin{proof}
If $\gamma_1(X_1|\mathcal{I}^{\mathbb{X}})=\infty$ a.s., then clearly~\eqref{zal_x3} and~\eqref{zal_x2} are satisfied so we can use Theorem~\ref{th4.2}. Moreover $I(\gamma_1(X_1|\mathcal{I}^{\mathbb{X}})<\infty)=0$ a.s. Hence applying part~(i) of Theorem~\ref{th4.2} we get ${\cal K}_t(k_{N(t)},a)/c(t)\xrightarrow{a.s.}0$ if $N(t)/c(t)\xrightarrow{a.s.}Z$. If $N(t)/c(t)\xrightarrow{d}Z$, then using part~(ii) of Theorem~\ref{th4.2} we obtain ${\cal K}_t(k_{N(t)},a)/c(t)\xrightarrow{d}0$, which is equivalent to ${\cal K}_t(k_{N(t)},a)/c(t)\xrightarrow{p}0$.
\end{proof}

Obviously, taking $k_n=n$, $n\ge 1$, in Theorems~\ref{th4.1}, \ref{th4.2} and Corollary~\ref{c4.1} we obtain results describing the asymptotic behavior of normalized numbers of near-maximum insurance claims. Note that the form of the sequence $(k_n,n\ge 1)$ does not affect the conclusion of Theorems~\ref{th4.1}, \ref{th4.2} and Corollary~\ref{c4.1} as long as~\eqref{warK} is satisfied. This means that the normalized numbers of near-maximum insurance claims exhibit the same asymptotic behavior as normalized numbers of claims  in a left neighborhood of  the $m$th largest claim ($m\in\mathbb{N}$) or even as normalized numbers of claims  in a left neighborhood of  the $m_n$th largest claim unless $m_n/n\to 0$ as $n\to\infty$.

We conclude this section with results concerning yet another quantities of interest - the sum of near-maxima:
\[
S_n(a)=\sum_{i=1}^n X_i\cdot I(X_{n:n}-a<X_i<X_{n:n}),
\]
and the total value of near-maximum insurance claims:
\[
\mathcal{S}_t(a)=\sum_{i=1}^{N(t)} X_i\cdot I(X_{N(t):N(t)}-a<X_i<X_{N(t):N(t)}).
\]
Again, we will consider a~more general case and will deal with the sum of observations  in a left neighborhood of  the $k_n$th order statistic:
\[
S(k_n,n,a)=\sum_{i=1}^n X_i\cdot I(X_{k_n:n}-a<X_i<X_{k_n:n})
\]
and the total value of claims  in a left neighborhood of  the $(n-k_n+1)$st largest insurance claim:
\[
\mathcal{S}_t(k_{N(t)},a)=\sum_{i=1}^{N(t)} X_i\cdot I(X_{k_{N(t)}:N(t)}-a<X_i<X_{k_{N(t)}:N(t)}).
\]

\begin{tw}\label{total}
Let $(k_n,n\ge 1)$ be a~sequence satisfying~\eqref{warK} and $\mathbb{X}=(X_n,n\ge~1)$ be a~strictly stationary sequence of rv's such that $\gamma_1(X_1|\mathcal{I}^{\mathbb{X}})=\infty$ a.s. Then
\begin{equation}\label{tot1}
\frac{S(k_n,n,a)}{X_{k_n:n}}\sim K(k_n,n,a)  \quad a.s.\hbox{ as } n\to\infty.
\end{equation}
If moreover $(N(t),t\ge 0)$ is a~process of non-negative integer valued rv's $N(t)$ such that $N(t)\xrightarrow{a.s.}\infty$, 
 then, as $t\to\infty$,
$$
\frac{\mathcal{S}_t(k_{N(t)},a)}{X_{k_{N(t)}:N(t)}}\sim \mathcal{K}_t(k_{N(t)},a)  \quad a.s.\hbox{ as } n\to\infty.
$$
\end{tw}
\begin{proof}
Following Li and Pakes (2001) we observe that
\[
(X_{k_n:n}-a)K(k_n,n,a)\le S(k_n,n,a)\le X_{k_n:n} K(k_n,n,a),\quad n\ge 1,
\]
which can be rewritten as
\begin{equation}\label{tot2}
    \left(1-\frac{a}{X_{k_n:n}}\right)K(k_n,n,a)   \le \frac{S(k_n,n,a)}{X_{k_n:n}}\le K(k_n,n,a),
\end{equation}
unless  $X_{k_n:n}>0$. If $\gamma_1(X_1|\mathcal{I}^{\mathbb{X}})=\infty$, then $X_{k_n:n}\xrightarrow{a.s.}\infty$, by Theorem~\ref{tw.2}. Letting $n\to\infty$ in~\eqref{tot2} gives~\eqref{tot1}. The second part of Theorem~\ref{total} follows now from Lemma~\ref{LEKM}.
\end{proof}

The next result is a simple consequence of Theorem \ref{total} and Corollaries \ref{con3.1} and \ref{c4.1}.
\begin{tw}\label{totalBis}
Under the assumptions of Corollary \ref{con3.1},
$$
 \frac{S(k_n,n,a)}{nX_{k_n:n}}\xrightarrow{a.s.}0\textrm{ as }n\to\infty.
$$
If moreover Condition \ref{cond_A} is satisfied and $N(t)/c(t)\xrightarrow{a.s.}Z$, then
$$
 \frac{\mathcal{S}_t(k_{N(t)},a)}{c(t)X_{k_{N(t)}:N(t)}}\xrightarrow{a.s.}0    \textrm{ as }t\to\infty.
$$
\end{tw}

\medskip
Examples of strictly stationary sequences $\mathbb{X}=(X_n,n\ge 1)$ satisfying $\gamma_1(X_1|\mathcal{I}^{\mathbb{X}})=\infty$ will be given in the sequel in Subsection~\ref{subsec:Sums}.

\section{Examples}
\label{sec:Examples}
In this section we apply results derived in the paper to some special classes of strictly stationary sequences.

\subsection{Strictly stationary and ergodic processes}
\label{subsec:Ergodic}

\begin{tw}\label{stat_ergodic}
Let $\mathbb{X}=(X_n,n\ge 1)$ be a~strictly stationary and ergodic sequence of rv's and $(k_n,n\ge 1)$ be a~sequence satisfying~\eqref{warK}.
\begin{description}
 \item[(i)]
 If $\gamma_1^{X_1}=\infty$ then~\eqref{zbiega_0} holds. If moreover Condition~\ref{cond_A} and~\eqref{ac4.1} are fulfilled, then~\eqref{tc4.1} holds.

\medskip
 \item[(ii)]
If $\gamma_1^{X_1}<\infty$ and
\begin{equation}\label{a5.1.1}
    \mathbb{P}(X_1=\gamma_1^{X_1})=0\textrm{ and }\mathbb{P}(X_1=\gamma_1^{X_1}-a)=0,
\end{equation}
then
\begin{equation}\label{t5.1.0}
   K(k_n,n,a)/n\xrightarrow{a.s.}\mathbb{P}(X_1>\gamma_1^{X_1}-a)\textrm{ as }n\to\infty. 
\end{equation}
Suppose further that Condition~\ref{cond_A} is satisfied. Then 
\begin{equation}
\label{t5.1.1}
\bullet   \quad      \mathcal{K}_t(k_{N(t)},a)/c(t)\xrightarrow{a.s.}Z\cdot\mathbb{P}(X_1>\gamma_1^{X_1}-a)\textrm{ as }t\to\infty,  \quad  \quad   
    \end{equation}
\phantom{$\bullet   \quad \quad $ }    provided that
    \begin{equation}\label{z5.1.1}
        N(t)/c(t)\xrightarrow{a.s.}Z\textrm{ as }t\to\infty;
    \end{equation}
 \begin{equation}
\label{t5.1.2}
 \bullet   \quad      \mathcal{K}_t(k_{N(t)},a)/c(t)\xrightarrow{p}Z\cdot\mathbb{P}(X_1>\gamma_1^{X_1}-a)\textrm{ as }t\to\infty,  \quad \quad  \quad
    \end{equation}
\phantom{$\bullet   \quad \quad $ }     provided that
    \begin{equation}\label{z5.1.2}
         N(t)/c(t)\xrightarrow{p}Z\textrm{ as }t\to\infty;
    \end{equation}
   \begin{equation}
\label{t5.1.3}
 \bullet   \quad             \mathcal{K}_t(k_{N(t)},a)/c(t)\xrightarrow{d}Z\cdot\mathbb{P}(X_1>\gamma_1^{X_1}-a)\textrm{ as }t\to\infty, \quad \quad 
    \end{equation}
\phantom{$\bullet   \quad \quad $ }      unless
    \begin{equation}\label{z5.1.3}
         N(t)/c(t)\xrightarrow{d}Z\textrm{ as }t\to\infty.
    \end{equation}
\end{description}
\end{tw}

\begin{proof}
The ergodicity of~$(X_n,n\ge 1)$ ensures that the measure of any set $A\in\mathcal{I}^{\mathbb{X}}$ equals 0 or~1. As a~consequence, every $\mathcal{I}^{\mathbb{X}}$-measurable extended rv is almost surely constant. Since $\gamma_1(X_1|\mathcal{I}^{\mathbb{X}})$ is $\mathcal{I}^{\mathbb{X}}$-measurable, by Theorem~\ref{stalyKwantyl} we get
\begin{equation}\label{pth5.1.1}
    \gamma_1(X_1|\mathcal{I}^{\mathbb{X}})=\gamma_1^{X_1}\textrm{ a.s.}
\end{equation}
Hence~(i) follows from Corollaries~\ref{con3.1} and \ref{c4.1}.

It remains to prove (ii). Observe that
\[
\mathbb{P}(X_1>\gamma_1(X_1|\mathcal{I}^{\mathbb{X}})-a|\mathcal{I}^{\mathbb{X}})=\mathbb{P}(X_1>\gamma_1^{X_1}-a|\mathcal{I}^{\mathbb{X}})
\]
is $\mathcal{I}^{\mathbb{X}}$-measurable and hence almost surely constant. This gives
\begin{align}\label{pth5.1.2}
    \mathbb{P}(X_1>\gamma_1(X_1|\mathcal{I}^{\mathbb{X}})-a|\mathcal{I}^{\mathbb{X}})=\mathbb{E}(\mathbb{P}(X_1>\gamma_1^{X_1}-a|\mathcal{I}^{\mathbb{X}}))=\mathbb{P}(X_1>\gamma_1^{X_1}-a).
\end{align}

From~\eqref{a5.1.1} and~\eqref{pth5.1.1} we conclude that conditions~\eqref{zal_x3} and~\eqref{zal_x2} are satisfied. Therefore we can use Theorems~3.3 and~\ref{th4.2}. By~\eqref{pth5.1.2} the former gives
\[
K(k_n,n,a)/n\xrightarrow{a.s.}\mathbb{P}(X_1>\gamma_1(X_1|\mathcal{I}^{\mathbb{X}})-a|\mathcal{I}^{\mathbb{X}})=\mathbb{P}(X_1>\gamma_1^{X_1}-a)\quad\textrm{a.s.},
\]
while the latter yields \eqref{t5.1.1}, \eqref{t5.1.2} and \eqref{t5.1.3}, provided that \eqref{z5.1.1}, \eqref{z5.1.2}  and \eqref{z5.1.3} hold, respectively.
\end{proof}

\begin{rem}
Condition~\eqref{a5.1.1} is satisfied whenever $X_1$ has a~continuous cumulative distribution function.
\end{rem}

\begin{rem}
Theorem~\ref{stat_ergodic} has a~quite wide range of applications, because the class of strictly stationary and ergodic sequences of rv's is broad. A~list of examples of members of this class can be found, for example, in Dembi\'nska (2014, Remark~2.1). Here we only briefly recall that every sequence of independent and identically distributed rv's is strictly stationary and ergodic. Moreover all linear processes belong to this class. In turn, the family of linear processes includes many processes of interest, like, for instance, stationary autoregressive-moving average processes or Gaussian processes with absolutely continuous spectrum.
\end{rem}

\begin{rem}
Some weaker versions of parts of Theorem~\ref{stat_ergodic} are known in the literature. For example, Pakes and Steutel (1997, Theorem 5.1) showed that~\eqref{t5.1.0} holds with $k_n=n$, $n\ge 1$, and ,,a.s'' replaced by ,,p'' provided that $\gamma_1^{X_1}<\infty$ and $\mathbb{X}=(X_n,n\ge 1)$ is a~sequence of independent rv's with a~common continuous cumulative distribution function. Assuming further that Condition~\ref{cond_A} and~\eqref{z5.1.2} are satisfied, Li and Pakes (2001) proved~\eqref{t5.1.2} with $k_n=n$, $n\ge 1$. A~special case of Theorem~5.1~(ii) with $c(t)=t$, $t>0$, and for sequences $(X_n,n\ge 1)$ of independent and identically distributed rv's is given in Dembi\'nska (2012a, Theorems~1 and~3 with $A=(0,a))$. It is also worth pointing out that a~weaker version of some part of Theorem~5.1~(ii) can be deduced from Proposition~2.8 of Hashorva~(2003).
\end{rem}

\subsection{Random shift and scaling of strictly stationary and ergodic processes}
\label{subsec:ShiftScaling}

In this subsection, we consider two sequences of rv's, $\mathbb{R}=(R_n,n\ge 1)$ and $\mathbb{S}=(S_n,n\ge 1)$, defined as follows
\[
R_n=X_n+U\textrm{ and }S_n=V\cdot X_n,\ n\ge 1,
\]
where $\mathbb{X}=(X_n,n\ge 1)$ is a~strictly stationary and ergodic sequence of rv's, $U$ is an rv and $V$ is a~positive rv. Because now we will work with three sequences of rv's, to avoid confusion, we again will use notation with superscripts. More pricesely, we will write $K^{\mathbb{X}}(k_n,n,a)$ and $\mathcal{K}^{\mathbb{X}}_t(k_{N(t)},a)$ to indicate that the counting rv's arise from the sequence $\mathbb{X}=(X_n,n\ge 1)$.

\begin{tw}\label{niesk}
If $\gamma_1^{X_1}=\infty$ and the sequence $(k_n.n\ge 1)$ satisfies~\eqref{warK}, then
\[
K^{\mathbb{R}}(k_n,n,a)/n\xrightarrow{a.s.}0\textrm{ and }K^{\mathbb{S}}(k_n,n,a)/n\xrightarrow{a.s.}0\textrm{ as }n\to\infty.
\]
If moreover Condition~\ref{cond_A} and~\eqref{ac4.1} hold, then, as $t\to\infty$,
\begin{align*}
  \mathcal{K}^{\mathbb{R}}_t(k_{N(t)},a)/c(t)\xrightarrow{a.s.}0\textrm{ and }\mathcal{K}^{\mathbb{S}}_t(k_{N(t)},a)/c(t)\xrightarrow{a.s.}0\\
  \left(\mathcal{K}^{\mathbb{R}}_t(k_{N(t)},a)/c(t)\xrightarrow{p}0\textrm{ and }\mathcal{K}^{\mathbb{S}}_t(k_{N(t)},a)/c(t)\xrightarrow{p}0\right).
\end{align*}
\end{tw}
\begin{proof}
Obviously, sequences $\mathbb{R}=(R_n,n\ge 1)$ and $\mathbb{S}=(S_n,n\ge 1)$ are strictly stationary. Therefore if we prove that
\begin{equation}\label{pth5.2.1}
    \gamma_1(R_1|\mathcal{I}^{\mathbb{R}})=\infty\textrm{ a.s. and }\gamma_1(S_1|\mathcal{I}^{\mathbb{S}})=\infty\textrm{ a.s.},
\end{equation}
the assertion will follow from Corollaries~\ref{con3.1} and~\ref{c4.1}. By Theorem~\ref{tw.2},
\begin{equation}\label{pth5.2.2}
    R_{k_n:n}\xrightarrow{a.s.}\gamma_1(R_1|\mathcal{I}^{\mathbb{R}})\textrm{  and }S_{k_n:n}\xrightarrow{a.s.}\gamma_1(S_1|\mathcal{I}^{\mathbb{S}}).
\end{equation}
On the other hand, Buraczy\'nska and Dembi\'nska (2018, Section~5) showed that
\begin{equation}\label{pth5.2.3}
    R_{k_n:n}\xrightarrow{a.s.}\gamma_1^{X_1}+U\textrm{  and }S_{k_n:n}\xrightarrow{a.s.}V\gamma_1^{X_1}.
\end{equation}
From~\eqref{pth5.2.2} and~\eqref{pth5.2.3} we see that
\begin{equation}\label{pth5.2.4}
    \gamma_1(R_1|\mathcal{I}^{\mathbb{R}})=\gamma_1^{X_1}+U\textrm{ a.s.  and }\gamma_1(S_1|\mathcal{I}^{\mathbb{S}})=V\gamma_1^{X_1}\textrm{ a.s.},
\end{equation}
which combined with the assumption that $\gamma_1^{X_1}=\infty$, gives~\eqref{pth5.2.1}. The proof is complete.
\end{proof}

The remainder of this subsection is devoted to the case of $\gamma_1^{X_1}<\infty$. Then the almost sure convergence of $K^{\mathbb{R}}(k_n,n,a)/n$ follows  from Theorem~\ref{stat_ergodic}.
\begin{tw}\label{zbieznosc_R}
Let $\gamma_1^{X_1}<\infty$, \eqref{a5.1.1} hold and $(k_n,n\ge 1)$ be a~sequence satisfying~\eqref{warK}. Then
\begin{equation}\label{cth5.3.1}
    K^{\mathbb{R}}(k_n,n,a)/n\xrightarrow{a.s.}\mathbb{P}(X_1>\gamma_1^{X_1}-a)\textrm{ as }n\to\infty.
\end{equation}
If moreover Condition~\ref{cond_A} is fulfilled, then
\[
N(t)/c(t)\xrightarrow{a.s.}Z\textrm{ implies }\mathcal{K}^{\mathbb{R}}_t(k_{N(t)},a)/c(t)\xrightarrow{a.s.}Z\cdot\mathbb{P}(X_1>\gamma_1^{X_1}-a),
\]
and the above implication still holds if we replace the almost sure convergence by convergence in probability or convergence in disribution.
\end{tw}
\begin{proof}
Since $R_{k_n:n}=X_{k_n:n}+U$, $n\ge 1$, we get
\begin{align*}
\frac{K^{\mathbb{R}}(k_n,n,a)}{n}&=\frac{\sum_{i=1}^n I(X_{k_n:n}+U-a<X_i+U<X_{k_n:n}+U)}{n}=\frac{K^{\mathbb{X}}(k_n,n,a)}{n}\\
&\xrightarrow{a.s.}\mathbb{P}(X_1>\gamma_1^{X_1}-a)\textrm{ as }n\to\infty,
\end{align*}
by Theorem~\ref{stat_ergodic}. Thus~\eqref{cth5.3.1} is shown. The rest of the proof goes along the same lines as the proof of Theorem~\ref{th4.2}.
\end{proof}

The last theorem concerns the sequence of rv's $\mathbb{S}=(S_n,n\ge 1)$.
\begin{tw}\label{zbieznosc_S}
Let $\gamma_1^{X_1}<\infty$ and $(k_n,n\ge 1)$ satisfies~\eqref{warK}. If
\begin{equation}\label{d5.4.1}
    \mathbb{P}(X_1=\gamma_1^{X_1})=\mathbb{P}(X_1+\frac{a}{V}=\gamma_1^{X_1})=0,
\end{equation}
then
\begin{equation}\label{cth5.4.0}
    K^{\mathbb{S}}(k_n,n,a)/n\xrightarrow{a.s.}\mathbb{P}(X_1+\frac{a}{V}>\gamma_1^{X_1}|\mathcal{I}^{\mathbb{S}})\textrm{ as }n\to\infty.
\end{equation}
If moreover Condition~\ref{cond_A} holds, then, as $t\to\infty$,
\begin{equation}\label{c5.4.1}
  \mathcal{K}^{\mathbb{S}}_t(k_{N(t)},a)/c(t)\xrightarrow{a.s.}Z\cdot\mathbb{P}(X_1+\frac{a}{V}>\gamma_1^{X_1}|\mathcal{I}^{\mathbb{S}}),  
\end{equation}
provided that
\begin{equation}\label{a5.4.1}
   N(t)/c(t)\xrightarrow{a.s.}Z\textrm{ as }t\to\infty.
\end{equation}
Furthermore, ``a.s.'' in~\eqref{c5.4.1} can be replaced by ``$\, p$''  and ``$d$"
if~\eqref{a5.4.1} is replaced by ``$N(t)/c(t)\xrightarrow{p}Z$'' and ``$N(t)/c(t)\xrightarrow{d}Z$ and $\mathbb{P}(X_1+\frac{a}{V}>\gamma_1^{X_1}|\mathcal{I}^{\mathbb{S}})$ is degenerate'', respectively.
\end{tw}
\begin{proof}
The definition of $(S_n,n\ge 1)$, assumptions~\eqref{a5.4.1} and~\eqref{pth5.2.4} guarantee that
\[
\mathbb{P}(S_1=\gamma_1(S_1|\mathcal{I}^{\mathbb{S}}))=0\textrm{ and }\mathbb{P}(S_1=\gamma_1(S_1|\mathcal{I}^{\mathbb{S}})-a)=0.
\]
Therefore we can apply Theorems~\ref{tw.5} and~\ref{th4.2} to the strictly stationary sequence $\mathbb{S}=(S_n,n\ge 1)$. Using these results and noting that~\eqref{pth5.2.4} gives
\[
\mathbb{P}(S_1>\gamma_1(S_1|\mathcal{I}^{\mathbb{S}})-a|\mathcal{I}^{\mathbb{S}})= \mathbb{P}(X_1>\gamma_1^{X_1}-\frac{a}{V}|\mathcal{I}^{\mathbb{S}})\quad\textrm{a.s.}
\]
we finish the proof.
\end{proof}

\subsection{Total value of claims near the $(n-k_n+1)$st largest insurance claim }
\label{subsec:Sums}

From the proofs of Theorems~\ref{stat_ergodic} and~\ref{niesk} we see that
\begin{itemize}
    \item 
    if $\mathbb{X}=(X_n,n\ge 1)$ is a~strictly stationary and ergodic sequence of rv's such that $\gamma_1^{X_1}=\infty$, then $\gamma_1(X_1|\mathcal{I}^{\mathbb{X}})=\infty$ a.s.;
    \item
    if $\mathbb{R}=(R_n,n\ge 1)$ and $\mathbb{S}=(S_n,n\ge 1)$ are sequences defined at the beginning of Subsection~\ref{subsec:ShiftScaling} and $\gamma_1^{X_1}=\infty$, then $\gamma_1(R_1|\mathcal{I}^{\mathbb{R}})=\infty$ a.s. and $\gamma_1(S_1|\mathcal{I}^{\mathbb{S}})=\infty$ a.s.
\end{itemize}
Therefore Theorems \ref{total} and \ref{totalBis}, describing the  long-term behavior of  total value of claims  in a left neighborhood of  the $(n-k_n+1)$st largest insurance claim, can be applied to any of the above mentioned sequences of rv's.

\bigskip
\bigskip
\noindent {\bf Acknowledgments}  

The work was supported  by Warsaw University of Technology under Grants nos. 504/ 03327/1120 and 504/ 03293/1120, and  by Polish National Science Centre under Grant no. 2015/19/B/ST1/03100.

\end{document}